\documentclass[11pt]{article}
\usepackage{amsfonts}
\usepackage{amsthm}

\theoremstyle{plain}
\newtheorem{theorem}{Theorem}[section]
\newtheorem{lemma}[theorem]{Lemma}

\begin{document}
\title{A Peak Point Theorem\\ for Uniform Algebras on Real-Analytic Varieties}
\author{John T. Anderson\\Department of Mathematics and Computer Science\\College of the Holy Cross\\Worcester, MA 01610 \and Alexander J. Izzo
\\Department of Mathematics and Statistics
\\Bowling Green State University
\\Bowling Green, OH 43403
\\aizzo@math.bgsu.edu
\\812-855-8081}
\date{}
\maketitle

\begin{abstract}
It was once conjectured that if $A$ is a uniform algebra on its maximal ideal space $X$, and if each point of $X$ is  a peak point for $A$, then $A = C(X)$.  This peak-point conjecture was disproved by Brian Cole in 1968.   Here we establish a peak-point theorem for uniform algebras generated by real-analytic functions on real-analytic varieties, generalizing previous results of the authors and John Wermer.\footnotetext[1]{2010 {\it Mathematics Subject Classification\/}. 32A38, 32A65, 46J10, 46J15}
\end{abstract}

\vskip -1.5 true in
\centerline{\footnotesize\it Dedicated to the memory of Andr\'e Boivin} 
\vskip 1.5 truein

\section{Introduction} \label{intro}

Let
$X$ be a compact metric space, and let $C(X)$ be the algebra of all continuous complex-valued functions on $X$ with the supremum norm
$ \|f\|_{X} = \sup\{ |f(x)| : x \in X \}$.  A \emph{uniform algebra} $A$ on $X$ is a closed subalgebra of $C(X)$ that contains the constant functions and separates
the points of $X$.  A central problem in the study of uniform algebras is to characterize $C(X)$ among the uniform algebras on $X$.  We consider the following two conditions:
\begin{description}
\item (i) the maximal ideal space $\frak{M}_{A}$ of $A$ is $X$, i.e., every non-zero
multiplicative linear functional on $X$ is given by point evaluation at a point of $X$;
 \item (ii) each point of $X$ is a peak point for $A$, i.e., given $x \in X$ there exists $f \in A$ with $f(x) = 1$ and $|f(y)| < 1$ for all $y \in X \setminus \{x\}$.                     \end{description}
Both (i) and (ii) are necessary conditions for $A = C(X)$.  It was once conjectured that (i) and (ii) together were sufficient to conclude that $A = C(X)$.  A counterexample to this peak-point conjecture was produced by Brian Cole in 1968  \cite{Co} (or see the Appendix to \cite{browder}, or \cite{stout_book}, chapter~3, section 19).
Other counterexamples have since been exhibited, including ones in which the underlying subset $X$ is contained in a complex Euclidean space $\mathbb{C}^n$ and the algebra $A$ is generated by holomorphic functions.  Richard Basener \cite{basener} constructed a subset $X$ of  the unit sphere in $\mathbb{C}^2$ such that if $A = R(X)$ is the closure in $C(X)$ of the rational functions holomorphic near $X$, then $A$ satisfies (i) and (ii), but $A \neq C(X)$.  Here (i) is equivalent to the statement that $X$ is rationally convex, i.e., for each point $z\notin X$ there exists a polynomial (in the complex coordinate functions) whose zero set contains $z$ and is disjoint from $X$.  Note that the peak-point condition (ii) is automatically satisfied when $X$ is a subset of the unit sphere and $A$ is an algebra containing the polynomials.
Basener \cite{basener}  also constructed a counterexample in which $X=M$ is a smooth three-sphere in $\mathbb{C}^6$ and $A=P(M)$ is the closure in $C(M)$ of the polynomials.
Here (i) is equivalent to the condition that $M$ is polynomially convex, i.e., for each point $z \notin M$ there exists a polynomial $P$ with $|P(z)| > \|P\|_{M}$.
Alexander Izzo \cite{izzo_counterexample} subsequently constructed a counterexample in which $X = M$ is a smooth solid torus (a three-dimensional manifold with boundary) lying in the unit sphere in $\mathbb{C}^5$ and again $A = P(M)$.  Lee Stout (\cite{stout_book2}, Theorem 6.5.20) constructed a modification with the smooth solid torus replaced by a smooth three-manifold lying in the unit sphere in $\mathbb{C}^6$ and diffeomorphic to the product of a circle with a two-sphere.

Despite the failure of the conjecture, it turns out that in a number of settings of interest, one can establish peak-point results characterizing $C(X)$.
For example, Anderson and Izzo \cite{anderson_izzo_two_manifolds} proved that if $X = M$ is a compact differentiable two-dimensional manifold with boundary, and $A$ is a uniform algebra generated by $C^1$-smooth functions and satisfying (i) and (ii), then $A = C(M)$.  Note that Basener's counterexample on a smooth three-sphere in $\mathbb{C}^6$ mentioned above shows that the corresponding statement is false for three-manifolds, even in the context of polynomial approximation.  However, Anderson, Izzo and Wermer \cite{anderson_izzo_wermer_three_dimensions} showed that if $M$ is a \emph{real-analytic} three-dimensional manifold with boundary in $\mathbb{C}^n$, $X$ is a compact subset of $M$ with boundary $\partial X$ a two-manifold of class $C^{1}$, and $A = P(X)$, then (i) and (ii) suffice to imply that $A = C(X)$.  The same authors later proved \cite{anderson_izzo_wermer_varieties} that if $V$ is a compact real-analytic variety in $\mathbb{C}^n$ of arbitrary dimension, then (i) and (ii) imply $P(V) = C(V)$.  Remarkably, this last result holds without the peak-point hypothesis: Lee Stout \cite{stout_varieties} showed that if $V$ is a compact real-analytic variety in $\mathbb{C}^n$, $A = \mathcal{O}(V)$ is the algebra of functions holomorphic in a neighborhood (dependent on the function) of $V$, and $V$ is holomorphically convex (i.e., $\frak{M}_{\mathcal{O}(V)} = V$), then $A = C(V)$.  In particular, this implies that if $V$ is polynomially convex (respectively rationally convex), then $P(V) = C(V)$ (respectively $R(V) = C(V)$).

In a slightly different vein, the present authors \cite{anderson_izzo_smooth_manifolds} established some results on uniform algebras generated by $C^1$-smooth functions on a smooth
manifold with boundary of arbitrary dimension. Some of these will be used here.  For a general survey on peak-point theorems to 2011, see \cite{izzo_survey}.

Our goal here is to establish a peak-point theorem for certain compact subsets $K$ of a real-analytic variety $V$ where the algebra $A$ is generated by real-analytic functions.  Since we are interested in a general uniform algebra $A$, we consider real-analytic varieties in real Euclidean space $\mathbb{R}^n$; of course the results will apply to real-analytic varieties in $\mathbb{C}^n$ with the usual identifications.  
In the special case when $V$ is a three-dimensional manifold, a slightly stronger result was proved by the present authors in \cite{anderson_izzo_smooth_manifolds}.

Throughout the paper, given a real-analytic subvariety $V$ of an open set $\Omega \subset \mathbb{R}^n$, and given a subset $K$ of $V$, whenever we say 
\lq\lq a collection $\Phi$ of functions real-analytic on $K$,\rq\rq\ we mean that to each member $f$ of $\Phi$ there corresponds a neighborhood of $K$ in $\mathbb{R}^n$, that may depend on the function $f$, to which $f$ extends to be real-analytic.  Also, $\partial K$ will denote the boundary of $K$ relative to $V$, and ${\rm int}(K)$ will denote the interior of $K$ relative to $V$.

The precise statement of our main result is as follows.

\begin{theorem} \label{main_theorem} Let $V$ be a real-analytic subvariety of an open set $\Omega \subset \mathbb{R}^n$, and let $K$ be a compact subset of $V$ such that $\partial K$ is a real-analytic subvariety of $V$.  Let $A$ be a uniform algebra on $K$ generated by a collection $\Phi$ of functions real-analytic on $K$.  Assume that $A$ satisfies conditions (i) and (ii) above.  Then $A = C(K)$.
\end{theorem}

Note that the peak-point hypothesis is necessary here, in contrast with Stout's result: if $V$ is the complex plane, $K$ the closed unit disk, and $A = P(K)$ is the algebra of functions holomorphic on the interior of $K$ and continuous on its closure, then $A$ satisfies (i) but not (ii), and of course $A \neq C(K)$.  Theorem 1.1 extends both results of Anderson, Izzo and Wermer mentioned above in two ways: (a) rather than restricting to the algebra generated by polynomials, we take a general algebra $A$ generated by real-analytic functions, and (b) we obtain a result on certain compact subsets of a variety $V$ whose dimension is arbitrary.

As is typical in these theorems, we argue using duality: it suffices to prove that if $\mu$ is a measure supported on $K$ with $\mu \in A^{\perp}$, i.e., such that
\[ \int f d\mu = 0 \mbox{ for all }f \in A, \]
then $\mu = 0$.  In fact, to prove Theorem 1.1, we will prove

\begin{theorem} \label{support_theorem} Let $V$ be a real-analytic subvariety of an open set $\Omega \subset \mathbb{R}^n$, and let $K$ be a compact subset of $V$.  Let $A$ be a uniform algebra on $K$ generated by a collection $\Phi$ of functions real-analytic on $K$.  Assume that $A$ satisfies (i) and (ii) above.  Then $\mu \in A^{\perp}$ implies\/ ${\rm supp}(\mu) \cap {\rm int}(K) = \emptyset$. \end{theorem}

Theorem \ref{main_theorem} is an easy consequence of Theorem \ref{support_theorem} (see the beginning of section \ref{proof} for the proof).  In section \ref{prelims} we collect some preliminary lemmas and comment on the general outline of the (quite technical) proof of Theorem \ref{support_theorem}, which is presented in section \ref{proof}.

\section{Preliminaries} \label{prelims}

Let $\Sigma$ be a real-analytic variety in the open set $W \subset \mathbb{R}^n$, that is, for each point $p \in W$ there is a finite set $\mathcal{F}$ of functions real-analytic in a neighborhood $N$ of $p$ in $W$ with $\Sigma $ the common zero set of $\mathcal{F}$ in $N$. The following definitions and facts are standard; see for example \cite{narasimhan}. We let $\Sigma_{\mbox{\tiny reg}}$ denote the set of points $p \in \Sigma$ for which there exists a neighborhood $N$ of $p$ in $\mathbb{R}^n$ such that $\Sigma \cap N$ is a regularly
imbedded real-analytic submanifold of $N$ of some dimension $d := d(p)$.   This dimension $d(p)$ is locally constant on $\Sigma_{\mbox{\tiny reg}}$.  The dimension of $\Sigma,$ denoted by $\mbox{dim}(\Sigma)$, is defined to be the largest such $d(p)$ as $p$ ranges over the regular points of $\Sigma$.  The singular set of $\Sigma$, denoted by $\Sigma _{\mbox{\tiny sing}}$, is the complement in $\Sigma$ of $\Sigma_{\mbox{\tiny reg}}$.  If $\Sigma^{\prime} \subset \Sigma$ is a real-analytic subvariety of some open set $W^{\prime} \subset W$ and $\Sigma^{\prime}$ has empty interior relative to $\Sigma$, then $\mbox{dim}(\Sigma^{\prime}) < \mbox{dim}(\Sigma)$.  Both $\Sigma$ and $\Sigma_{\mbox{\tiny sing}}$ are closed in $W$.  Although $\Sigma_{\mbox{\tiny sing}}$ may not itself be a subvariety of $W$, it is locally contained in a proper subvariety of $\Sigma$: for each $p \in \Sigma _{\mbox{\tiny sing}}$, there is a real-analytic subvariety $Y$ of an open neighborhood $N$ of $p$ such that $\Sigma_{\mbox{\tiny sing}} \cap N \subset Y$  and $\mbox{dim}(Y) < \mbox{dim}(\Sigma)$.

If $M$ is an $m$-dimensional manifold of class $C^1$ and
$\Phi$ is a collection 
of functions that are $C^{1}$ on
$M$, we define the {\em exceptional set $M_{\Phi}$ of $M$ relative to $\Phi$}  by
\begin{equation} M_{\Phi} = \{ p \in M : df_{1} \wedge \ldots \wedge df_{m}(p) = 0 \mbox{ for all }m \mbox{-tuples } (f_{1}, \ldots , f_{m}) \in  \Phi^m \}. \label{eq:exceptional} \end{equation}
If $\Sigma$ is a real-analytic variety and $\Phi$ is a collection of functions real-analytic on $\Sigma$ (i.e., each function in $\Phi$ extends to be real-analytic in a neighborhood (depending on the function) of $\Sigma$), then $\Sigma_{\Phi}$ is defined to be the set of all points $p \in \Sigma_{\mbox{\tiny reg}}$ such that in a neighborhood of $p$, the set  $\Sigma_{\mbox{\tiny reg}}$ is an $m$-dimensional manifold $M$, and $p \in M_{\Phi}$ as defined in (\ref{eq:exceptional}).

\begin{lemma} \label{exceptional_set_variety} Let $\Sigma$ be a real-analytic variety in an open set $W \subset \mathbb{R}^n$, and $\Phi$ a collection of functions real-analytic on $\Sigma$.  Then $\Sigma_{\Phi}$ is a subvariety of $W \setminus \Sigma _{\mbox{\rm\tiny sing}}$. \end{lemma}

\begin{proof} Suppose $p \in \Sigma_{\mbox{\tiny reg}}$.  We may choose a neighborhood $N$ of $p$ in $\mathbb{R}^n$ so that $\Sigma_{\mbox{\tiny reg}} \cap N$ is an $m$-dimensional real-analytic submanifold of $N$, for some $m$,  and local coordinates $t_{1}, \ldots ,t_{2n}$ in $N$ so that $t_{1}, \ldots, t_{m}$ are real-analytic local coordinates on $\Sigma_{\mbox{\tiny reg}} \cap N$.  Then $\Sigma_{\Phi} \cap N$ is the common zero set of the determinants of $( \partial f_{i} / \partial t_{j})_{i,j=1}^m$ over all $m$-tuples $f_{1}, \ldots , f_{m}$ of elements of $\Phi$.  Each of these determinants extends to a real-analytic function on a neighborhood of $\Sigma_{\mbox{\tiny reg}} \cap N$ in $N$, depending on the functions $f_{1}, \ldots, f_{m}$.  The set of all such determinants generates an ideal in the ring of germs $\mathcal{O}_{p}$ of real-analytic functions at $p$.  Since $\mathcal{O}_{p}$ is Noetherian (see \cite{narasimhan}, chapter 2, Theorem 4), this ideal is finitely generated.  
We may therefore choose finitely many functions from the ideal so that these functions are real-analytic in a fixed neighborhood $N'\subset N$ of $p$ and so that the common zero set of these functions is equal to $\Sigma_{\Phi} \cap N'$. \end{proof}

\begin{lemma}[Lemma 2.2 of \cite{anderson_izzo_smooth_manifolds}]\label{exceptional_set_empty_int} Let $M$ be an $m$-dimensional differentiable submanifold of $\mathbb{R}^n$ with boundary of class $C^{1}$.
Let $K$ be a compact subset of $M$, and let $A$ be a uniform algebra on $K$ generated by a collection $\Phi$ of functions of class $C^1$ in a neighborhood of $K$ in $M$.  Assume that $A$ satisfies (i) and (ii).  Then $M_{\Phi} \cap K$ has empty interior in $M$. \end{lemma}

\begin{lemma}[Lemma 2.3 of \cite{anderson_izzo_smooth_manifolds}]\label{inheritance_lemma} Let $A$ be a uniform algebra on a compact Hausdorff space $X$ satisfying (i) and (ii).  If $Y$ is a closed subset
of $X$, then $\overline{A|Y}$ also satisfies (i) and (ii). \end{lemma}

We will make use of a recent result of Izzo that will enable us to reduce approximation on a variety to approximation on the union of the exceptional set and the singular set of the variety.  This type of theorem has a long history, going back to work of John Wermer \cite{wermer} and Michael Freeman \cite{freeman} in the 1960's - for a detailed account, see \cite{izzo_approx_on_manifolds}.

\begin{theorem}[Izzo, \cite{izzo_approx_on_manifolds}] \label{approximation_theorem}
Let $A$ be a uniform algebra on a compact
Hausdorff space $X$, and suppose that the maximal ideal space of $A$ is $X$.
Suppose also that $E$ is a closed subset of $X$ such that $X \setminus E$
is an $m$-dimensional manifold and such that
\begin{enumerate}
\item for each point $p \in X \setminus E$ there are functions
$f_1, \ldots, f_m$ in $A$ that are $C^1$ on $X \setminus E$ and satisfy $df_1
\wedge \ldots \wedge df_m (p) \neq 0$, and
\item  the functions in $A$ that are $C^1$ on $X \setminus E$ separate points
on $X$.
  \end{enumerate}
Then $A= \{g \in C(X) : g | E \in A|E \}$.  \end{theorem}

The general idea of the proof of Theorem \ref{support_theorem} is to use Theorem \ref{approximation_theorem} to reduce approximation on a variety $V$ to approximation on the union of the singular set of $V$ and the exceptional set of the algebra $A$.  As we have noted, the singular set of $V$ is locally contained in a proper subvariety of $V$, i.e., a variety of dimension strictly less than the dimension of $V$.  The exceptional set, in the presence of the peak-point hypothesis, is also a proper subvariety of the regular set of $V$, by Lemma~\ref{exceptional_set_variety} combined with Lemma~\ref{exceptional_set_empty_int}.  One would like to then use induction to reduce approximation (i.e., the support of a putative annihilating measure $\mu$) on $V$ to approximation on a sequence of varieties of decreasing dimension, until the dimension is zero (i.e., the variety is a discrete set), and then to conclude that the support of an annihilating measure must be empty.  However, it is not obvious that the \emph{union} of the exceptional set and the singular set, even locally, must itself be contained in a subvariety of 
$V$ of dimension less than that of $V$ (although we have no counterexample).  Could, for example, the exceptional set accumulate at every point of the singular set?

To get around this difficulty, we treat the exceptional set and singular set separately, introducing a filtration of $V$ into exceptional sets and singular sets of decreasing dimensions, then first showing by induction on decreasing dimension of the exceptional sets that the support of any annihilating measure $\mu$ must lie 
in the singular set.  We then use induction again on a decreasing sequence of singular sets to reduce the support of $\mu$ to the empty set.

\section{Proof of Theorems 1.1 and 1.2} \label{proof}

We first indicate how Theorem \ref{main_theorem} can be obtained from
Theorem~\ref{support_theorem}.  Let the variety $V$, the compact set $K \subset V$, and the algebra $A$ be as in Theorem~\ref{main_theorem}.  If $\mu \in A^{\perp}$, Theorem~\ref{support_theorem} implies that $\mbox{supp}(\mu) \subset \partial K$.  Apply Theorem~\ref{support_theorem} with $V$ replaced by $\partial K$ and $K$ replaced by $\partial K$ also.  Lemma~\ref{inheritance_lemma} implies that $\overline{A|\partial K}$ satisfies (i) and (ii). Note that ${\rm int}(\partial K)$ relative to $\partial K$ is $\partial K$. Therefore Theorem~\ref{support_theorem} implies that $\mbox{supp}(\mu) \cap \partial K = \emptyset$, and hence $\mbox{supp}(\mu) = \emptyset$, so $\mu \equiv 0$.  This establishes Theorem~\ref{main_theorem}.

We now turn to the proof of Theorem \ref{support_theorem}, beginning with a general construction.

Let $\Sigma$ be a real-analytic variety in the open set $W \subset \mathbb{R}^n$, and let $\Phi$ be a collection of functions real-analytic on $\Sigma$.  We define inductively subsets $\Sigma_{k}$ of $\Sigma$ such that $\Sigma_{0} = \Sigma$, and for $k \geq 1$, $\Sigma_{k}$ is a real-analytic subvariety of
\[ W_{k} := W \setminus \bigcup_{j=0}^{k-1} \; (\Sigma_{j})_{\mbox{\tiny sing}} \]
defined by
\[ \Sigma_{k} = (\Sigma_{k-1})_{\Phi}. \]
Note that by definition, $\Sigma_{k} \subset (\Sigma_{k-1})_{\mbox{\tiny reg}}$.
We will refer to the varieties $\Sigma_{k}$ as the {\em E-filtration of $\Sigma$ in $W$ with respect to $\Phi$}, and to the sets $(\Sigma_{k})_{\mbox{\tiny sing}}$ as the {\em S-filtration of $\Sigma$ in $W$} (\emph{E} for exceptional, \emph{S} for singular).

\begin{lemma} \label{dimension_lemma} With $V,\Omega,K, A, \Phi$ as in Theorem \ref{support_theorem}, suppose that $W$ is an open subset of $\Omega$  and $\Sigma \subset {\rm int}(K) \cap W$ is a real-analytic subvariety of $W$.  Let $\{ \Sigma_{k} \}$ be the E-filtration of $\Sigma$ in $W$ with respect to $\Phi$.  Then for each $k$, the dimension of $\Sigma_{k}$ is no more than $\mbox{dim} (\Sigma) - k$. \end{lemma}

\begin{proof} Let $d = \mbox{dim}(\Sigma)$. The proof is by induction on $k$. The result is clear when $k = 0$.  Suppose we have shown for some $k$ that $\mbox{dim}(\Sigma_{k})\leq d - k$.  Fix $p \in (\Sigma_{k})_{\mbox{\tiny reg}}$, and let $U$ be a smoothly bounded
neighborhood of $p$ in $(\Sigma_{k})_{\mbox{\tiny reg}}$ with $\overline{U} \subset (\Sigma_{k})_{\mbox{\tiny reg}}$.  We may assume that
$U$ has constant dimension (by induction, no more than $d - k$) as a submanifold of
$\mathbb{R}^n$.  Lemma~\ref{inheritance_lemma} implies that (i) and (ii) hold with $A$ replaced by $\overline{A|{\overline{U}}}$.  We may therefore apply
Lemma~\ref{exceptional_set_empty_int} taking $M =\overline{U}$ and
replacing $A$ with $\overline{A|{\overline{U}}}$. The conclusion implies that $\Sigma_{k+1} = (\Sigma_{k})_{\Phi}$ has no interior in $U$.  Since $p$ was arbitrary, we conclude that
\[ \mbox{dim}(\Sigma_{k+1}) \leq \mbox{dim}(\Sigma_{k}) - 1 \leq d - (k+1) \].
By induction, the proof is complete. \end{proof}

Note that Lemma \ref{dimension_lemma} implies, with $d = \mbox{dim}(\Sigma)$,  that $\Sigma_{d}$ is a zero-dimensional variety, i.e., is a discrete set and hence, in particular, is at most countable.

Let $B(p,r)$ denote the open ball of radius $r$ centered at $p \in \mathbb{C}^n$.

\begin{lemma} \label{support_in_S_filtration} With $V,\Omega,K, A, \Phi$ as in Theorem \ref{support_theorem}, assume $p \in {\rm int}(K)$ and $\mu\in A^\perp$.  If $r > 0$ is such that $B(p,r) \cap V \subset K$, and there is a real-analytic
$d$-dimensional subvariety $\Sigma \subset V$ of $B(p,r)$ with\/ $\mbox{\rm supp }(\mu) \cap B(p,r) \subset \Sigma$, then\/ $\mbox{\rm supp }(\mu) \cap B(p,r)$ is contained in the S--filtration of $\Sigma$, i.e., $\mbox{\rm supp }(\mu) \cap B(p,r)\subset
\bigcup\limits_{k=0}^{d-1} (\Sigma_{k})_{\mbox{\rm\tiny sing}}$.
\end{lemma}

\begin{proof}
We will show by induction on $L$ that
\[ \mbox{supp}(\mu) \cap B(p,r) \subset \bigcup_{k=0}^{L} (\Sigma_{k})_{\mbox{\tiny sing}} \cup \Sigma_{L+1} \]
for each $L$, $0 \leq L \leq d-1$.
This suffices as the hypothesis that each point of $K$ is a peak point for $A$ implies that $\mu\in A^\perp$ has no point masses, and hence $\mu(\Sigma_d)=0$ since $\Sigma_d$ is at most countable.

For the $L = 0$ case, let $X = (K \setminus B(p,r)) \cup \Sigma_{0} $ and let $E =(K \setminus B(p,r)) \cup (\Sigma_{0})_{\mbox{\tiny sing}} \cup \Sigma_{1}$.  Note that both $X$ and $E$ are closed.  We want to show that $\mbox{supp}(\mu)\subset E$.   By Lemma \ref{inheritance_lemma}, the maximal ideal space of
$\overline{A|X}$ is $X$.  Note that $X \setminus E = \Sigma_{0} \setminus \bigl((\Sigma_{0})_{\mbox{\tiny sing}} \cup \Sigma_{1}\bigr)$ satisfies the hypotheses of Theorem~2.4.  Therefore by Theorem \ref{approximation_theorem}, if $g \in C(K)$ vanishes on $E$, then $g|X$ belongs to $\overline{A|X}$.  Since by hypothesis $\mbox{supp}(\mu) \subset X$, we get that $\int_{K} g \; d\mu = 0$ for each $g \in C(K)$ vanishing on $E$, and this implies that $\mbox{supp}(\mu) \subset E$, as desired.

The general induction step is similar: assuming the result for some $0\leq L < d-1$, we set
\[ X = (K \setminus B(p,r)) \cup \bigcup_{k=0}^{L} (\Sigma_{k})_{\mbox{\tiny sing}} \cup \Sigma_{L+1}, \]
\[ E = (K \setminus B(p,r)) \cup \bigcup_{k=0}^{L+1} (\Sigma_{k})_{\mbox{\tiny sing}} \cup \Sigma_{L+2}, \]
Both $X$ and $E$ are closed.  Noting that the induction hypothesis implies that $\mbox{supp}(\mu) \subset X$, we apply Theorem \ref{approximation_theorem} to $\overline{A|X}$ as above, to conclude that $\mbox{supp}(\mu) \subset E$.  \end{proof}

\begin{lemma} \label{support_in_boundary} With $V,\Omega,K, A, \Phi$ as in Theorem \ref{support_theorem}, assume $p \in {\rm int}(K)$ and $\mu\in A^\perp$.  Assume also that there exist $r > 0$ with $B(p,r) \cap V \subset {\rm int}(K)$ and a real-analytic subvariety $\Sigma \subset V$ of $B(p,r)$ with\/ $\mbox{\rm supp}(\mu) \cap B(p,r)$ contained in the S--filtration of $\Sigma$ in $B(p,r)$.  Then there exists $r^{\prime} > 0$ such that\/ $\mbox{\rm supp}(\mu) \cap B(p,r^{\prime}) = \emptyset$. \end{lemma}

\begin{proof} We apply induction on the dimension of $\Sigma$.  If $\mbox{dim}(\Sigma) = 0$, then $\Sigma$ is discrete.  Since each point of $K$ is a peak point for $A$, the measure $\mu \in A^{\perp}$ has no point masses, and so $|\mu|(B(p,r)) = |\mu|(\Sigma) = 0$.  Now suppose the conclusion of the Lemma holds whenever $\mbox{dim}(\Sigma) < d$.  If $\mbox{dim}(\Sigma) = d$, let $\Sigma_{0}, \ldots, \Sigma_{d}$ be the \emph{E}-filtration of $\Sigma$ in $B(p,r)$.  (Recall that $\Sigma_{d}$ is discrete.) By induction on $L$ we will show that
\begin{equation} \mbox{supp}(\mu) \cap B(p,r) \subset \bigcup_{k=0}^{d-1-L} (\Sigma_{k})_{\mbox{\tiny sing}}. \label{eq:supp} \end{equation}
for $L = 0, \ldots , d-1$.  The case $L = 0$ is the hypothesis of the Lemma.  Assume we have established (\ref{eq:supp}) for some $L$, $0 \leq L < d-1$.
To show that (\ref{eq:supp}) holds with $L$ replaced by $L + 1$, we must show that $\mbox{supp}(\mu) \cap (\Sigma_{d-1-L})_{\mbox{\tiny sing}} = \emptyset$.  Fix $q \in (\Sigma_{d-1-L})_{\mbox{\tiny sing}}$.  By construction of the $\Sigma_{k}$, there exists $s > 0$ so that $B(q,s) \subset B(p,r)$ and $B(q,s) \cap (\Sigma_{k})_{\mbox{\tiny sing}} = \emptyset$ for all $k < d-1-L$. Therefore the induction hypothesis implies that $\mbox{supp}(\mu) \cap B(q,s) \subset (\Sigma_{d-1-L})_{\mbox{\tiny sing}}$.  Replacing $s$ by a smaller positive number if necessary, we may assume that there is a real-analytic subvariety $Y$ of $B(q,s)$ with $(\Sigma_{d-1-L})_{\mbox{\tiny sing}} \subset Y \subset V$ and $\mbox{dim}(Y) < \mbox{dim}(\Sigma_{d-1-L}) \leq L + 1 < d$ (the next-to-last inequality following from Lemma \ref{dimension_lemma}).  By Lemma \ref{support_in_S_filtration}, $\mbox{supp}(\mu) \cap B(q,s)$ is contained in the \emph{S}-filtration of $Y$ in $B(q,s)$.  Now note that our induction hypothesis on dimension implies that the conclusion of Lemma \ref{support_in_boundary} holds with $\Sigma$ replaced by $Y$, since $\mbox{dim}(Y) < d$.  We conclude that there exists $s'>0$ such that $\mbox{supp}(\mu) \cap B(q,s') = \emptyset$.  Since $q  \in (\Sigma_{d-1-L})_{\mbox{\tiny sing}}$ was arbitrary, this shows that $\mbox{supp}(\mu) \cap (\Sigma_{d-1-L})_{\mbox{\tiny sing}} = \emptyset$ and completes the proof that (\ref{eq:supp}) holds for $L = 0, \ldots , d-1$.

Finally, the case $L = d-1$ of (\ref{eq:supp}) asserts that $\mbox{supp}(\mu) \cap B(p,r) \subset (\Sigma_{0})_{\mbox{\tiny sing}}$.   We may choose $t$ with $0 < t< r$ and a subvariety $Y \subset V$ of $B(p,t)$ so that $(\Sigma_{0})_{\mbox{\tiny sing}} \subset Y$ and $\mbox{dim}(Y) < \mbox{dim}(\Sigma) = d$.  By Lemma \ref{support_in_S_filtration}, $\mbox{supp}(\mu) \cap B(p,t)$ is contained in the \emph{S}-filtration of $Y$ in $B(p,t)$.  Again applying our induction hypothesis on dimension, we conclude that there exists $r'>0$ such that $\mbox{supp}(\mu) \cap B(p,r') = \emptyset$.  This completes the proof.  \end{proof}

We can now finish the proof of Theorem \ref{support_theorem}.  Let $\mu \in A^\perp$ be given.  Fix $p \in {\rm int}(K)$, and $r > 0$ such that $B(p,r) \cap V \subset \mbox{int}(K)$.  Taking $\Sigma = V \cap B(p,r)$ in Lemma \ref{support_in_S_filtration}, we see that $\mbox{supp}(\mu) \cap B(p,r)$ is contained in the \emph{S}--filtration of $V \cap B(p,r)$.  By Lemma \ref{support_in_boundary}, there is an $r'>0$ such that $\mu$ has no support in $B(p,r')$, concluding the proof.

We end with two remarks.  First, as we pointed out at the end of section \ref{prelims}, the role of the peak-point hypothesis serves to reduce the size of the exceptional set (see Lemma \ref{exceptional_set_empty_int}).  In Stout's theorem \cite{stout_varieties}, discussed in section \ref{intro}, absent a peak-point hypothesis, a similar role is played by a theorem of Diederich and Fornaess \cite{diederich-fornaess}, which states that a compact real-analytic variety in $\mathbb{C}^n$ contains no non-trivial germ of a complex-analytic variety.  Is there, for more general underlying spaces/algebras, a suitable generalization of the Diederich-Fornaess theorem?

Finally, the most general setting for a peak-point theorem such as Theorem \ref{main_theorem} would appear to be in the category of real-analytic spaces. We believe our methods could be extended to that context.

\medskip
{\bf Acknowledgement\/}: This paper was completed while the second author was a visitor at Indiana University.  He would like to thank the Department of Mathematics for its hospitality.

\end{document}